\documentclass[11pt,a4paper,reqno]{amsart}
\usepackage{geometry}
\usepackage{amsmath,amsthm,amssymb}
\usepackage[utf8]{inputenc}
\usepackage{microtype}

\newtheorem{theorem}{Theorem}[section]
\newtheorem{proposition}[theorem]{Proposition}
\newtheorem{lemma}[theorem]{Lemma}
\newtheorem{corollary}[theorem]{Corollary}

\numberwithin{equation}{section}

\usepackage{cleveref}
\crefname{theorem}{Theorem}{Theorems}
\crefname{proposition}{Proposition}{Propositions}
\crefname{lemma}{Lemma}{Lemmas}
\crefname{corollary}{Corollary}{Corollaries}
\crefname{example}{Example}{Examples}
\crefname{definition}{Definition}{Definitions}
\crefname{remark}{Remark}{Remarks}

\crefname{enumi}{}{}
\crefname{enumii}{}{}
\crefname{enumiii}{}{}
\crefname{equation}{}{}

\renewcommand{\Re}{\operatorname{Re}}
\renewcommand{\Im}{\operatorname{Im}}

\usepackage{enumerate}

\newcommand{\RR}{\mathbb{R}}
\newcommand{\CC}{\mathbb{C}}

\newcommand{\E}{\mathcal{E}}
\newcommand{\J}{J[\varrho]}
\newcommand{\Tor}{\mathrm{Tor}}

\newcommand{\wba}{\bar{w}}

\newcommand{\zba}{\bar{z}}
\newcommand{\Zba}{\overline{Z}}
\newcommand{\Lba}{\overline{L}}
\newcommand{\Nba}{\overline{N}}

\newcommand{\kba}{\bar{k}}

\newcommand{\ob}{\bar{1}}

\begin{document}
\title{The CR umbilical locus of a real ellipsoid in~$\mathbb{C}^2$}
\author{Duong Ngoc Son}
\address{Faculty of Fundamental Sciences, PHENIKAA University, Hanoi 12116, Vietnam}
\email{son.duongngoc@phenikaa-uni.edu.vn}
\begin{abstract} This paper concerns the CR umbilical locus of a real ellipsoid in  $\CC^2$, the set of points at which the ellipsoid can be osculated by a biholomorphic image of the sphere up to 6th order. Huang and Ji proved that this locus is non-empty. Ebenfelt and Zaitsev proved that, for ellipsoids that are sufficiently ``closed'' to the sphere, the locus actually contains a ``stable'' curve of umbilical points. Foo, Merker and Ta later provided an explicit curve that is contained in it. The main result in this paper exhibits that the umbilical locus is the union of the curves mentioned above and a non-trivial real variety which is defined by two homogeneous real sextic equations of four real variables. When there are suitable pairs of equal semi-axes, one of them factors allowing us to determine the locus explicitly.	 
\end{abstract}
\date{September 4, 2022.}
\thanks{This work was begun while the author was at the University of Vienna, Austria. He gratefully acknowledges the support of the Austrian Science Fund (FWF): Projekt I4557-N}
\maketitle
\section{Introduction}
One of the most important biholomorphic invariants of nondegenerate real hypersurfaces in $\CC^2$ is the Cartan's umbilical tensor \cite{cartan1933geometrie}, denoted by  $Q$. It is a 6th order curvature tensor that characterizes the equivalence to the unit sphere. Precisely, if $Q$ vanishes on a neighborhood of a point $p$ on an analytic real hypersurface $M$, then $M$ is locally CR spherical (i.e. locally biholomorphically equivalent to the sphere) near $p$. If $Q$ vanishes at $p\in M$, then there is a biholomorphic image of the unit sphere with 6th order (or better) contact with $M$ at $p$ \cite{cartan1933geometrie,chern1974real}, and in this case, $p$ is called a CR \textit{umbilical point} of $M$. The umbilicality or non-umbilicality are biholomorphic invariant and related to other interesting geometric properties. For examples, Chern--Moser proved in their seminal paper \cite{chern1974real} that the stability group at a non-umbilical point consists of at most two elements. This fact leads to a construction of a vector field which is essentially invariant near each non-umbilical point. Chern--Moser also posed a question (which has not been fully resolved yet) about the existence of nowhere umbilical boundaries in $\CC^2$, see \cite{ebenfelt2018family} for a construction of such a boundary diffeomorphic to a three-torus.  Fefferman \cite{fefferman1976} constructed an example which exhibits Cartan chains spiralling in towards an umbilical point and asked (still open) if chains can only spiral in towards an umbilical point?

In the higher dimensional case, the notion of CR umbilical points is also well-defined (using the 4th order Chern--Moser tensor instead of the Cartan 6th order tensor). Webster \cite{webster2000holomorphic} proved, for the case $N\geq 3$, that a ``generic'' real ellipsoid in $\mathbb{C}^N$ has an empty CR umbilical locus (see also \cite{son2021semi} for a simpler proof). Hence, the higher dimensional version of the Chern--Moser question has an affirmative answer. In contrast, Huang--Ji \cite{huang2007every} proved that every real ellipsoid in $\CC^2$ admits umbilical points (i.e., real ellipsoids in $\CC^2$ do not answer Chern--Moser question). Further results on the existence or nonexistence of umbilical points on more general three-dimensional CR manifolds were obtained by Ebenfelt, Zaitsev, and the author in \cite{ebenfelt2017umbilical,ebenfelt2018family,ebenfelt2019new}. For instance, it has been proved that umbilical points must exist on complete circular boundaries in $\CC^2$ \cite{ebenfelt2017umbilical} and on various perturbations of the sphere (including real ellipsoids that are ``closed'' to the sphere) \cite{ebenfelt2019new}. In \cite{ebenfelt2018family}, Ebenfelt--Son--Zaitsev provided an explicit family of nowhere umbilical compact real hypersurfaces that are diffeomorphic to a three-torus, partly resolved the Chern--Moser question. For real ellipsoids, Foo--Merker--Ta \cite{foo2018parametric} improved Huang--Ji's result by providing an explicit parametric curve that is contained in the umbilical locus. However, in view of Monge's result \cite{monge1976} for two-dimensional ellipsoids in $\RR^3$, our understanding in CR case was not completely satisfactory: It was not known from the previous results if there are umbilical points other than those on the curve mentioned above.

The main aim of this paper is to determine the CR umbilical locus of an arbitrary real ellipsoid in $\mathbb{C}^2$. Recall that an ellipsoid $\mathcal{E} \subset \CC^2$ is a compact real hypersurface given by the equation
\begin{align*}
A x^2  + B y^2 +  C u^2 +  D v^2 = 1, \quad A, B, C, D \in \mathbb{R}_{>0}.
\end{align*}
Here, $(z,w) = (x + iy,u+iv)$ is the complex coordinate system of $\mathbb{C}^2$ which gives the identification $\CC^2 \cong \RR^4$. In this coordinate system, the last equation can be written as
\[
\alpha |z|^2 + \beta |w|^2 + \Re \left( a z^2 + b w^2\right) - 1 = 0,
\]
with $\alpha = (A+B) /2$, $\beta = (C+D)/2$, $a = (A-B)/2$, and $b = (C-D)/2$. By the biholomorphic transformation $(\tilde{z}, \tilde{w}) = (\sqrt{\alpha} z, \sqrt{\beta}w)$, we can assume that $\alpha = \beta = 1$. We then obtain a two-parameter family which can ``represent'' all biholomorphic equivalent classes of real ellipsoids. 
The main result of this paper is a description of their CR umbilical loci. To describe it, we write $Z = (z,w)$ for the coordinates of $\mathbb{C}^2$, $\rho_{ZZ}$ for the ``holomorphic Hessian'' of $\rho$, and $L = \rho_w \partial_z - \rho_z \partial_w$ (which restricts to a frame for $T^{1,0} \mathcal{E}$). Thus, for $\rho$ given by \eqref{e:elip1}, we have
\[
	\rho_{ZZ}(L,L) 
	=
	a\rho_w^2 +b \rho_z^2.
\]
For the ellipsoid $\mathcal{E}$, we also put $N = \rho_{\bar{z}} \partial_z + \rho_{\bar{w}}\partial _w$ (cf. Eq. \eqref{e:211}). From this, $\rho_{ZZ}(N,N)$ and $\rho_{ZZ}(N,L)$ can be defined similarly. Moreover, let $J[\rho]$ denote the Levi--Fefferman determinant, see \eqref{e:lf}. We are now ready to state our main result.
\begin{theorem}\label{thm:1}
Let $0\leqslant b \leqslant a <1$, $a\ne 0$, and let $\mathcal{E}\subset \mathbb{C}^2$ be the real ellipsoid defined by $\varrho  = 0$ where
\begin{equation}\label{e:elip1}
\varrho  := |z|^2 + |w|^2 + \Re(az^2 + bw^2) -1.
\end{equation}
Then the  CR umbilical locus of $\mathcal{E}$ is given by the union of the curves $\gamma_{1}$ and $\gamma_{2}$ (corresponding to the ``$+$'' and ``$-$'' signs in \eqref{e:f2} below) determined by
\begin{align}\label{e:f1}
z 
& =
\sqrt{\frac{a}{a+b}} \left(\frac{\sqrt{1-b}\, \cos(t)}{\sqrt{1+a}} + \frac{i\sqrt{1+b}\, \sin(t)}{\sqrt{1-a}}\right),\\ \label{e:f2}
w
& = \pm
\sqrt{\frac{b}{a+b}} \left(\frac{\sqrt{1-a}\, \sin(t)}{\sqrt{1+b}} - \frac{i\sqrt{1+a}\, \cos(t)}{\sqrt{1-b}}\right), \quad t\in [0, 2\pi),
\end{align}
and a non-trivial real variety $\mathcal{V} \subset \mathcal{E}$ given by the equation
\begin{align}\label{e:u2a}
0= -\frac{1}{2} \frac{ab\,\overline{\varrho _{ZZ}(L,L)}}{\J^3} - 2\frac{\overline{\varrho _{ZZ}(N,N)}}{\J^3} 
+ \frac{|\varrho _{ZZ}(L,L)|^2}{\J^4}   \notag  \\ 
- 4\frac{|\varrho _{ZZ}(N,L)|^2}{\J^4}
-\frac{5}{2}\frac{\overline{\varrho _{ZZ}(L,L)}\left(\varrho _{ZZ}(N,L)\right)^2}{\J^5},
\end{align}
which is the common intersection of $\mathcal{E}$ and two conical sextic real varieties with apexes at the origin.

Moreover,
\begin{enumerate}[(i)]
\item if $b=0$, then $\mathcal{V}$ is the union of the curves parametrized by
\begin{align}\label{c1}
z = \pm \sqrt{\frac{s_0}{(1+a)(1+a+s_0)}},\ 
w = \sqrt{\frac{1+a}{1+a+s_0}} \left(\cos(t) + i \sin (t)\right),
\end{align}
where $s_0$ is the unique positive root to the cubic
\begin{equation}
4s^3 + 8 (1+a) s^2 + (4+6a + 5a^2) s - 2a =0;
\end{equation}
\item if $b=a$, then the CR umbilical locus of $\mathcal{E}$ is the union of the curves having the following parameterizations
\begin{align}\label{e:16}
z 
& =
\frac{1}{\sqrt{1-a +\tau^2 (1+a)}}\left( \frac{\sqrt{1-a}\, \cos (t)}{\sqrt{1+a}}+ \frac{i\,\tau\,\sqrt{1+a}\, \sin (t)}{\sqrt{1-a}}\right), \\ \label{e:17}
w 
& = 
\frac{1}{\sqrt{1-a +\tau^2 (1+a)}}\left( \frac{\sqrt{1-a}\, \sin (t)}{\sqrt{1+a}} - \frac{i\,\tau\,\sqrt{1+a}\, \cos (t)}{\sqrt{1-a}}\right),
\end{align}
with $\tau \in \{-1, 1, -\sqrt{s_0}, \sqrt{s_0}\}$ and $s_0$ is the unique positive root of the cubic
\begin{equation}
(a^2+2a+4) s^3 + (4-a (19 a+34)) s^2 + (a (19 a-34)-4) s -a^2+ 2 a-4 =0.
\end{equation}
In this case, $\mathcal{V}$ consists of the curves corresponding to $\tau = \pm \sqrt{s_0}$.
\end{enumerate}
\end{theorem}
The proof of this theorem will be given in Section 4, based on Corollary~3.10.

As briefly mentioned earlier, the curves given by \eqref{e:f1} and \eqref{e:f2}, constituting a proper subset of the CR umbilical locus, were found earlier in \cite{foo2018parametric}. These curves turn out to be the solutions to a system of two homogeneous quadratic equations in four real variables. In the case $b=0$ and $a\ne 0$, the ellipsoid is a real hypersurface of revolution and these two curves becomes the one given by
\[
z = \frac{1}{\sqrt{1+a}} \cos (t) + \frac{i}{\sqrt{1-a}} \sin(t), \quad w = 0.
\]
The CR umbilicality of points along this curve follows directly from the observation that each of the rotations $(z,w) \mapsto (z, e^{it} w), t\in \RR,$ fixes every point along this curve, in view of the Chern--Moser normalization at a non-umbilical point \cite{chern1974real}, see also \cite[Section 5.6]{ebenfelt2019new}. Two newly discovered curves in \eqref{c1} (in the case $b=0$), \eqref{e:16} and \eqref{e:17} (for $\tau = \pm \sqrt{s_0}$ in the case $b=a$) are the solutions to two homogeneous sextic equations, which we are able to solve explicitly when $b= 0$ or $b=a$. In the general case, we are unable to obtain explicit solutions. But in view of the Abel--Ruffini impossibility theorem, it could be the case that those multivariate sextic equations, with generic values of $a$ and $b$, cannot be solved by radicals.

\section{Preliminaries}
\subsection{Pseudohermitian geometry}
Let $M\subset \CC^2$ be a real hypersurface. The complex structure on $\CC^2$ induces a CR structure on $M$ in a natural way. Namely, if we define $T^{1,0} M := \CC TM \cap T^{1,0} \CC^2$, then $T^{1,0}M$ is a CR structure bundle on $M$ of complex dimension one. If we define $H: = \Re T^{1,0}M$, then $H$ is a real two-dimensional sub-bundle of the real tangent bundle $TM$. If we suppose that $H$ is nondegenerate, in the sense that for arbitrary vector fields $X$ and $Y$ that locally span $H$, the Lie bracket $[X,Y]$ is everywhere transverse to $H$, then $H$ becomes a contact structure on $M$. If $\theta$ is a real 1-form such that $\ker \theta = H$, then the nondegeneracy of $H$ is equivalent to $\theta \wedge d\theta \ne 0$ everywhere and $\theta$ is called a contact form. In this case, the characteristic (or Reeb) field of $\theta$ is the unique vector field $T$ such that $T \rfloor d\theta = 0$ and $\theta(T) = 1$. The contact form $\theta$ is also called a pseudohermitian structure \cite{webster1978pseudo} and $(M,\theta)$ is called a pseudohermitian manifold (with the underlying CR structure understood). If $Z_1$ is any complex vector field locally spanning $T^{1,0}M$, the admissible 1-form dual to $Z_1$ is the unique 1-form $\theta^1$ such that the coframe $\{\theta_1,\theta^{\ob},\theta\}$ is dual to the frame $\{Z_1, Z_{\ob},T\}$ and
\begin{equation}
d\theta = i h_{1\ob} \theta^1 \wedge \theta^{\ob}
\end{equation}
for some real function $h_{1\ob}$.  Here $\theta^{\ob} := \overline{\theta^1}$ and $Z_{\ob} := \overline{Z_1}$ are the complex conjugations. We shall assume that $h_{1\ob} > 0$ and say that $M$ is strictly pseudoconvex. The Levi form is denoted by 
\begin{equation}
\langle U, \overline{V} \rangle 
=
h_{1\ob} U^1 V^{\ob} \quad
\text{for}\ U = U^1 Z_1\ \text{and}\ \overline{V} = V^{\ob} Z_{\ob}.
\end{equation}
The (complexified) Tanaka--Webster connection of $(M,T^{1,0}M,\theta)$ is the connection $\nabla$ on $\CC TM$ given in terms of a local frame $Z_1$ by
\begin{equation}
\nabla Z_1 = \omega_1{}^{1} \otimes Z_1,
\quad
\nabla Z_{\ob} = \omega_{\ob}{}^{\ob} \otimes Z_{\ob}, 
\quad 
\nabla T = 0,
\end{equation}
where the connection form $\omega_{1}{}^{1}$ is the complex 1-form uniquely determined by
\begin{equation}
d\theta^1 = \theta^1 \wedge \omega_{1}{}^{1} + A_{\ob}{}^1 \theta \wedge \theta^{\ob},
\quad
\omega_{1}{}^{1} + \omega_{\ob}{}^{\ob} = d\log h_{1\ob}.
\end{equation}
The function $A_{\ob}{}^1$ is the coefficient of the Webster torsion $A$:
\[
A Z_{\ob} := \mathrm{Tor}(T, Z_{\ob}) = A_{\ob}^1 Z_1.
\]
The structure equation for the Tanaka--Webster connection is \cite{webster1978pseudo}
\begin{equation}
d\omega_{1}{}^{1} = R h_{1\ob} \theta^{1} \wedge \theta^{\ob} + A_{1}{}^{\ob}{}_{,\ob} \theta^1 \wedge \theta - A_{\ob}{}^{1}{}_{,1}\theta^{\ob} \wedge \theta,
\end{equation}
where $R$ is the Webster scalar curvature, which is real-valued.

The Cartan tensor is a relative invariant of the CR structure on $M$ whose vanishing is necessary and sufficient for $M$ to be locally CR spherical \cite{cartan1933geometrie}. If the pseudohermtian structure fixed, the Cartan tensor can be interpreted as an endomorphism of $H$, written locally as \cite{cheng1990burns}
\begin{equation}\label{e:26}
Q = i Q_{1}{}^{\ob} \theta^1 \otimes Z_{\ob} - i Q_{\ob}{}^{1}\theta^{\ob} \otimes Z_1.
\end{equation}
Then $Q$ is a pseudohermitian invariant which is CR-covariant in the sense that if $\hat{\theta} = e^{u}\theta$ is another pseudohermitian structure, then the corresponding Cartan tensor $\hat{Q}$ with respect to $\hat{\theta}$ satisfies $\hat{Q} = e^{-2u}Q$ \cite{cheng1990burns}. The Cartan tensor can be computed from the scalar curvature, the torsion, and their covariant derivatives via the following well-known formula \cite[Lemma 2.2]{cheng1990burns}.
\begin{equation}\label{e:chenglee}
Q_{1}{}^{\ob}
=
\frac{1}{6} R_{,1}{}^{\ob}
+ \frac{i}{2}R A_{1}{}^{\ob}
- A_{1}{}^{\ob}{}_{,0}
- \frac{2i}{3}
A_{1}{}^{\ob}{}_{,\ob}{}^{\ob}.
\end{equation}
Here, the indices preceded by a comma indicate covariant derivatives and the $0$-index indicates the covariant derivative along the Reeb direction.
\subsection{Curvature and torsion of a real hypersurfaces in local coordinates}
The main purpose of this section is to collect some useful formulas for the Tanaka--Webster connection form as well as the scalar curvature and torsion. In general, these formulas have been computed by Li--Luk \cite{li--luk}, cf. \cite{webster1978pseudo}.

Let $M\subset \mathbb{C}^{2}$ be a nondegenerate real hypersurface defined by $\varrho = 0$ with $d\varrho \ne 0$ along $M$. Let $\theta: = \iota^{\ast}(i \bar{\partial} \varrho)$. Then $\theta$ is a pseudohermitian structure on $M$. Let
\begin{equation}
Z_1 = L : = \varrho _{w} \partial_z - \varrho _z \partial_w,
\end{equation}
be a basis for $T^{1,0}(M)$ and let $\theta^1$ be a dual admissible coframe, then
\begin{equation}
h_{1\ob}
=
\J,
\end{equation}
where $\J$ is the the Levi--Fefferman determinant (a.k.a bordered complex Hessian) of $\varrho $, namely,
\begin{equation}\label{e:lf}
\J 
=
-\det
\begin{pmatrix}
\varrho  & \varrho _{\zba} & \varrho _{\wba} \\
\varrho _{z} & \varrho _{z\zba} & \varrho _{z\wba}\\
\varrho _{w} & \varrho _{w\zba} & \varrho _{w\wba}
\end{pmatrix}.
\end{equation}
For simplicity, we assume that $\J > 0$ (that $\J \ne 0$ characterizes the nondegeneracy of $M$). For a basis of the ``normal'' bundle $N^{1,0}(M)$ in $T^{1,0} \mathbb{C}^2$, we use 
\begin{equation}\label{e:211}
N: = \J \xi
\end{equation}
where $\xi$ is the unique $(1,0)$-vector field determined by 
\[
\partial\varrho (\xi) = 1, \quad \bar{\partial}\partial \varrho  \rfloor \xi = 0.
\]
Explicit formulas for the connection form $\omega_{1}{}^{1}$ and the torsion form $\tau^1$ (also for general dimensional case) were given in \cite{li--luk}. We describe them as follows. First we write
\begin{equation}
\omega_{1}{}^{1} = \Gamma_{11}^1 \theta + \Gamma_{\ob 1}^1 \theta^{\ob} + \Gamma_{01}^1 \theta.
\end{equation}
Then these Christoffel symbols are given by 
\begin{proposition}\label{prop:2.1} Suppose $M$ is a real hypersurface in $\mathbb{C}^2$ defined by $\varrho  = 0$ with $\J >0$. Let $\theta = \iota^{\ast}(i\bar{\partial}\varrho )$. Then the Christoffel symbols of the Tanaka--Webster connection in the frame $Z_1 = L$ are
\begin{align}
\Gamma_{11}^{1} &= L \log \J,\\
\Gamma_{01}^1 & =
\frac{i}{\J} \left(N \log \J - 2\det \varrho _{Z\Zba}\right),
\\
\Gamma_{\ob 1}^1 &= 0,
\end{align}
and their conjugates, i.e.,  $\Gamma_{0\ob}^{\ob} = \overline{\Gamma_{01}^1}$ and so on. Moreover,
\begin{equation}\label{e:tors}
	iA_1{}^{\ob} Z_{\ob}
	=
	Z_1(\xi^{\ob}) \partial_{\zba} + Z_{1}(\xi^{\bar{2}}) \partial_{\wba}. \qedhere
\end{equation}
\end{proposition}
These formulas are essentially in \cite[Equation (2.20)]{li--luk} after a change of the local frame. In the case $n=1$, the proof is simpler and we provide it here for completeness.
\begin{proof}
First, we observe that the Reeb vector field is
\begin{equation}
T = i(\xi - \bar{\xi}) = \frac{i}{\J}[Z_1,Z_{\ob}].
\end{equation}
Therefore
\begin{equation}
\Tor(Z_1, Z_{\ob}):= \nabla_{Z_1} Z_{\ob} - \nabla_{Z_{\ob}} Z_1 - [Z_1, Z_{\ob}]
=
\Gamma_{1\ob}^{\ob} Z_{\ob} - \Gamma_{\ob 1}^{1} Z_{1} - i \J\, T.
\end{equation}
On the other hand, from Tanaka \cite{tanaka1975differential}, the torsion must satisfy
\begin{equation}
\Tor(Z_1,Z_{\ob})
=
i \langle Z_1, Z_{\ob}\rangle T.
\end{equation}
We immediately find that
\begin{equation}
\Gamma_{1\ob}^{\ob}
=
\Gamma_{\ob 1}^{1}
= 0.
\end{equation}
On the other hand, since the Levi-form is parallel, we have that
\begin{equation}
Z_1 \J
= 
Z_1 \cdot \langle Z_1, Z_{\ob} \rangle
=
\langle \nabla_{Z_1} Z_{1} , Z_{\ob} \rangle 
+
\langle Z_1 , \nabla_{Z_1} Z_{\ob} \rangle 
=
\J \Gamma_{1 1}^{1}.
\end{equation}
This proves that
\begin{equation}
\Gamma_{1 1}^1 = Z_1 \log \J.
\end{equation}
Finally, using
\begin{equation*}
A_1{}^{\ob} Z_{\ob}
=
\Tor(T,Z_{1})
=
\nabla_T Z_1 - \nabla_{Z_1} T - [T,Z_1]
=
\Gamma_{01}^1 Z_1 - [T,Z_1],
\end{equation*}
we find that,
\begin{equation}
[T,Z_1]
=
\Gamma_{01}^1 Z_1 - A_1{}^{\ob} Z_{\ob}.
\end{equation}
By explicitly computation of $[T,Z_1]$, we find that
\begin{align*}
-i[T,Z_1]
& =
[(\xi - \overline{\xi}), Z_1]\\
& =
\left(\xi(\varrho _w) - \bar{\xi}(\varrho _w) - Z_1(\xi^1) \right)\partial_z
-
\left(\xi(\varrho _z) - \overline{\xi}(\varrho _z) + Z_1(\xi^2)\right) \partial_w \\
& \quad -
\left(Z_1(\xi^{\ob})\partial_{\zba} + Z_1(\xi^{\bar{2}}) \partial_{\wba} \right).
\end{align*}
By taking the $(1,0)$-parts of both sides, we deduce (after some simplifications) the formula for $\Gamma^1_{01}$. Likewise, by taking the $(0,1)$-parts of both sides, we obtain \eqref{e:tors}. The proof is complete.
\end{proof}
For the curvature and torsion, we have the following
\begin{proposition}[Li--Luk \cite{li--luk}]\label{prop:ct} Under the notations above, the Webster scalar curvature of $(M,i\bar{\partial}\varrho )$ is given by
\begin{align}
R 
=
\frac{1}{\J}\left(2\det \varrho _{Z\Zba}  - N \log \J -  \Lba L \log \J\right),
\end{align}
and the torsion in the frame $Z_1$ is given by
\begin{equation}\label{eq:tor}
iA_{11} = \frac{\J L\left(\xi^{\ob}\right)}{\varrho _{\wba}},
\end{equation}
where $\J$ is the Levi--Fefferman determinant and $L:= \varrho_w \partial_z - \varrho_z \partial_w$.
\end{proposition}
These formulas above can be derived from \cite[Theorem~1.1 and (2.13)]{li--luk} and a simple change of holomorphic frame (for the component of the torsion). The formula for the torsion can also be derived from \eqref{e:tors}. We leave the details to the readers.
\section{The Cartan CR umbilical tensor on ``pluriharmonic perturbations'' of the sphere}
We compute the Cartan tensor for the manifold $M$ defined by
\begin{equation}\label{e:perbsph}
\varrho (z,w): = -1 + |z|^2 + |w|^2 + 2\Re (f(z,w)) = 0,
\end{equation}
where $f(z,w)$ is holomorphic. To this end, we shall compute four pseudohermitian invariants appearing on the right-hand side of \eqref{e:chenglee}. Our formula for the Cartan tensor in this special case will be simpler than those in general case  \cite{ebenfelt2019new,foo2018parametric}, as various ``mixed'' derivatives of $\varrho$ of orders greater than 2 vanish.

The pseudohermitian invariants on $M$ we shall consider are with respect to the contact form $\theta = \iota^{\ast}(i\bar{\partial}\varrho )$ and the holomorphic coframe $\theta^1$. We shall work with the holomorphic frame
\begin{equation}
Z_1 = L = \varrho _w \partial_z - \varrho _z\partial_w,
\quad
N = \varrho _{\zba}\partial_z + \varrho _{\wba} \partial_w
\end{equation}
with 
\begin{equation}
\varrho _z = \zba + f_z, \quad 
\varrho _w = \wba + f_w.
\end{equation}
We can verify that when restricted to $M$,
\begin{equation}
\J\bigl|_M = \varrho _{Z\Zba}(L, \Lba)\bigl|_M = |\varrho _z|^2 + |\varrho _w|^2.
\end{equation}
Here, $Z=(z,w)$ is the coordinates in $\mathbb{C}^2$, $\varrho _{Z\Zba}$ is the complex Hessian, so that for $X = x^j \partial_j$ and $\overline{Y} = y^{\kba} \partial_{\kba} $ (summation convention), we have
\[
\varrho _{Z\Zba} (X, \overline{Y}) = \varrho _{j\kba} x^j y^{\kba}.
\]
We also use similar and self-explanatory notations for the second, the third and the fourth orders (ordinary) partial derivatives $\varrho _{ZZ}$, $\varrho _{ZZZ}$, and $\varrho _{ZZZZ}$ which act on ordered pairs, triples, or quadruples of vectors in $T^{(1,0)} \CC^2$, respectively. They certainly depend on the chosen coordinates of $\mathbb{C}^2$.

Let's start with the following formulas for the curvature and torsion of $M$ defined by \eqref{e:perbsph}.
\begin{lemma}[Gauß equations]\label{lem:rt} If $\varrho $ is given as in \eqref{e:perbsph}, then the torsion $A_{11}$ in the frame $Z_1=L:=\varrho _w \partial_z - \varrho _z \partial_w$ and the Webster scalar curvature $R$ of $(M,i\bar{\partial}\varrho )$ are given by
\begin{equation}\label{e:tor}
iA_{11} = 
\frac{\varrho _{ZZ}(L,L)}{\J}
\end{equation}
and 
\begin{equation}\label{e:scal}
R
= 
\frac{2}{\J}
-\frac{|\varrho _{ZZ}(L,L)|^2}{\J^3},
\end{equation}
respectively.
\end{lemma}
\begin{proof} 
These formulas follow easily from Li--Luk's results in general case \cite{li--luk} as in Proposition~\ref{prop:ct}. In fact, observe that
\begin{equation}
\Lba L \J = \frac{1}{\J}\left(|\varrho _{ZZ}(L,L)|^2 + |L\J|^2\right) - N \J.
\end{equation}
Both sides are polynomial expressions in the derivatives of $\varrho $.
Hence
\begin{equation}
\Lba L \log \J
= 
-\frac{\varrho _{ZZ}(N,N)}{\J} + \left|\frac{\varrho _{ZZ}(L,L)}{\J}\right|^2.
\end{equation}
Then \cref{e:scal} follows easily. We can also derive \cref{e:tor} from the formula for the torsion. We omit the details.
\end{proof}
Equations \eqref{e:scal} and \eqref{e:tor} are manifestations of Gau\ss{} equations for semi-isometric CR immersions in \cite{son2021semi}, see also \cite{reiter2021chern}. 
\subsection{The term $A_{11,0}$}
To compute $A_{11,0}$ we shall need a formula for the Christoffel symbol $\Gamma^1_{01}$.
\begin{lemma} If $\varrho $ is given by \cref{e:perbsph}, then Christoffel symbol
\begin{equation}\label{e:g101}
i\Gamma_{01}^1
= \frac{2}{\J} - \frac{\varrho _{ZZ}(N,N)}{\J^2}.
\end{equation}
\end{lemma}
\begin{proof}
The proof follows from Proposition \ref{prop:2.1} via a direct calculation.
\end{proof}
\begin{lemma} If $\varrho $ is given by \cref{e:perbsph}, then
\begin{align}\label{e:a110}
A_{11,0}
=
\frac{2 \det \varrho _{ZZ}}{\J} + \frac{2\varrho _{ZZ}(L,L)+\varrho _{ZZZ}(N,L,L)}{\J^2} \notag \\ + \frac{\varrho _{ZZ}(L,L)\left(\overline{\varrho _{ZZ}(N,N)} - 3\varrho _{ZZ}(N,N)\right)}{\J^3}.
\end{align}
\end{lemma}
\begin{proof}
Differentiating along $N$ and $\Nba$ directions, we find that
\begin{align*}
N(\varrho _{ZZ}(L,L))
=
\J \det \varrho _{ZZ} +
\varrho _{ZZZ}(N,L,L),
\end{align*}
and
\begin{align*}
\Nba (\varrho _{ZZ}(L,L))
= 
2\varrho _{ZZ}(L,L) -  \J \det \varrho _{ZZ}.
\end{align*}
Since $T = i\J^{-1} (N - \overline{N})$ and $N\J = \varrho _{ZZ}(N,N)$, we have 
\begin{align}
TA_{11}
& =
T\left(-\frac{i\varrho _{ZZ}(L,L)}{\J}\right) \notag  \\ 
& =
\frac{2\det \varrho _{ZZ}}{\J} + \frac{\varrho _{ZZZ}(N,L,L) - 2\varrho _{ZZ}(L,L)}{\J^2}\notag \\
& \quad + \frac{\varrho _{ZZ}(L,L)(\overline{\varrho _{ZZ}(N,N)}-\varrho _{ZZ}(N,N))}{\J^3}. 
\end{align}
Plugging this and \cref{e:g101} into the formula
\begin{equation}
A_{11,0}
=
TA_{11} - 2 \Gamma_{01}^1 A_{11},
\end{equation}
we obtain the desired identity.
\end{proof}
\subsection{The term $A_{11,}{}^{1}{}_1$}
The following three lemmas can be proved by direct calculations. We leave the details to the readers.
\begin{lemma}\label{lem44} If $\varrho $ is given by \cref{e:perbsph}, then the identity
\begin{equation}
\J^2 \det \varrho _{ZZ} + \left(\varrho _{ZZ}(N,L)\right)^2 = \varrho _{ZZ}(L,L) \varrho _{ZZ}(N,N)
\end{equation}
holds on $M$.
\end{lemma}
\begin{lemma}[Mainardi equation] If $\varrho $ is given by \cref{e:perbsph}, then 
\begin{equation}
\Lba\left(\varrho _{ZZ}(L,L)\right) + 2 \varrho _{ZZ}(N,L) = 0.
\end{equation}
\end{lemma}
\begin{lemma}  If $\varrho $ is given by \cref{e:perbsph}, then
\begin{align}
L\J & = \varrho _{ZZ}(N,L),\\
LL\J 
& =
\varrho _{ZZ}(L,L) - (\J + \varrho ) \det \varrho _{ZZ}.
\end{align}
\end{lemma}
\begin{proposition} With the notations as above, it holds that
\begin{align}\label{e:a1111}
i A_{11,}{}^{1}{}_{1}
=
-\frac{2\det \varrho _{ZZ}}{\J}
- \frac{2\varrho _{ZZ}(L,L) + 2\varrho _{ZZZ}(N,L,L)}{\J^2} \notag \\
+ \frac{\varrho _{ZZ}(L,L)(\overline{\varrho _{ZZ}(N,N)}+6 \varrho _{ZZ}(N,N)) - \varrho _{ZZZ}(L,L,L)\overline{\varrho _{ZZ}(N,L)}}{\J^3} \notag \\
+ \frac{\varrho _{ZZ}(L,L)(3|\varrho _{ZZ}(N,L)|^2 - |\varrho _{ZZ}(L,L)|^2)}{\J^4}.
\end{align}
\end{proposition}
\begin{proof}
Since $\Gamma_{\ob 1}^{1} = 0$ and $h^{1\ob} = \J^{-1}$, we have that
\begin{equation}
A_{11,}{}^{1}
=
h^{1\ob} \Lba (A_{11})
=
i\left(\frac{2\varrho _{ZZ}(N,L)}{\J^2} + \frac{\varrho _{ZZ}(L,L) \overline{\varrho _{ZZ}(N,L)}}{\J^3}\right).
\end{equation}
Differentiating along $L$, we have
\begin{equation}
L\left(\varrho _{ZZ}(N,L)\right)
=
\varrho _{ZZ}(L,L) 
-
\J \det \varrho _{ZZ} + \varrho _{ZZZ}(N,L,L),
\end{equation}
and
\begin{equation}
L\left(\overline{\varrho _{ZZ}(N,L)}\right)
=
-\overline{\varrho _{ZZ}(N,N)} + \frac{|\varrho _{ZZ}(L,L)|^2 + |\varrho _{ZZ}(N,L)|^2}{\J}.
\end{equation}
Thus
\begin{align}
-i L \left(A_{11,}{}^{1}\right)
=
-\frac{2\det \varrho _{ZZ}}{\J} + \frac{2\varrho _{ZZ}(L,L) + 2\varrho _{ZZZ}(N,L,L)}{\J^2}\notag \\
- \frac{4\left(\varrho _{ZZ}(N,L)\right)^2}{\J^3} + \frac{\varrho _{ZZZ}(L,L,L) \overline{\varrho _{ZZ}(N,L)}}{\J^3}  \notag \\
- \frac{\varrho _{ZZ}(L,L) \overline{\varrho _{ZZ}(N,N)}}{\J^3} - \frac{\varrho _{ZZ}(L,L)(|\varrho _{ZZ}(L,L)|^2 - 2|\varrho _{ZZ}(N,L)|^2)}{\J^4}.
\end{align}
Plugging these into the formula
\begin{equation}
iA_{11,}{}^{1}{}_{1} = i L \left(A_{11,}{}^{1}\right) - 2i \left(\frac{\varrho _{ZZ}(N,L)}{\J}\right) A_{11,}{}^{1},
\end{equation}
and simplifying the result using Lemma~\ref{lem44}, we complete the proof.  
\end{proof}
\subsection{The term $R_{,11}$} Differentiating the scalar curvature, using Li--Luk's formula, we obtain
\begin{lemma} If $\varrho $ is given as in \eqref{e:perbsph}, then
\begin{small}
\begin{align} \label{e:r11}
R_{,11}
=
-\frac{4\det \varrho _{ZZ}}{\J} - \frac{2\varrho _{ZZ}(L,L) + 2\varrho _{ZZZ}(N,L,L)}{\J^2} \notag \\
- \frac{3\left(\det\varrho _{ZZ}\right)|\varrho _{ZZ}(L,L)|^2}{\J^3}  + \frac{2\varrho _{ZZ}(L,L)(3\varrho _{ZZ}(N,N) - \overline{\varrho _{ZZ}(N,N)})}{\J^3} \notag \\
+ \frac{\overline{\varrho _{ZZ}(L,L)}\left(\varrho _{ZZZZ}(L,L,L,L)-3 \varrho _{ZZZ}(L,L,\varrho _{ZZ} \cdot L)\right)}{\J^3} + \frac{4\varrho _{ZZZ}(L,L,L) \overline{\varrho _{ZZ}(N,L)}}{\J^3} \notag \\
+\frac{\varrho _{ZZ}(L,L)(5|\varrho _{ZZ}(L,L)|^2-12 |\varrho _{ZZ}(N,L)|^2)}{\J^4} + \frac{3|\varrho _{ZZ}(L,L)|^2 \varrho _{ZZZ}(N,L,L)}{\J^4} \notag \\
+ \frac{7\,\overline{\varrho _{ZZ}(L,L)}\varrho _{ZZ}(N,L)\varrho _{ZZZ}(L,L,L) }{\J^4} - \frac{15|\varrho _{ZZ}(L,L)|^2 \left(\varrho _{ZZ}(N,L)\right)^2}{\J^5}.
\end{align}
\end{small}

The right-hand side involves the derivatives of $\varrho $ up to 4th order.
\end{lemma}
\begin{proof} Differentiating \cref{e:scal}, we have
\begin{equation}
R_{,1}
=
- \frac{2}{\J} L\log \J  + \frac{3}{\J^3} |\varrho _{ZZ}(L,L) |^2 L \log \J  - \frac{1}{\J^3}  L\left(  |\varrho _{ZZ}(L,L)|^2\right).
\end{equation}
Differentiating one more time and using $\Gamma^1_{11} = L\log \J$, we have
\begin{align*}
R_{,11}
& = 
L\left(R_{,1} \right) - (L\log \J) R_{,1}\\
& = 
-\frac{2}{\J} \left(LL\log \J - 2(L \log \J)^2\right) \\
& \quad + \frac{3}{\J^3} |\varrho _{ZZ}(L,L)|^2 \left(LL\log \J - 4(L \log \J)^2\right) \\
& \quad + \frac{1}{\J^3} \left(7(L \log \J)( L\left( |\varrho _{ZZ}(L,L)|^2 \right) - L|\varrho _{ZZ}(L,L)|^2\right).
\end{align*}
To expand the expression further, we use
\begin{equation}
LL \log \J
=
\frac{\varrho _{ZZ}(L,L)}{\J}
+ \frac{\varrho _{ZZZ}(N,L,L)}{\J} - \frac{\varrho _{ZZ}(L,L)\varrho _{ZZ}(N,N)}{\J^2 },
\end{equation}
and 
\begin{align*}
L \Lba \left(\varrho _{ZZ}(L,L)\right)
=
-2 \varrho _{ZZ}(L,L) + 2\, \J\det \varrho _{ZZ} - 2\,\varrho _{ZZZ}(N,L,L),
\end{align*}
which can be checked directly. Thus, we have
\begin{equation}
L\left(|\varrho _{ZZ}(L,L)|^2\right) = -2 \varrho _{ZZ}(L,L) \overline{\varrho _{ZZ}(N,L)} - \varrho _{ZZZ}(L,L,L) \overline{\varrho _{ZZ}(L,L)},
\end{equation}
and hence
\begin{align}
LL\left(|\varrho _{ZZ}(L,L)|^2\right)
=
-4 \varrho _{ZZZ}(L,L,L) \overline{\varrho _{ZZ}(N,L)}- 2 \varrho _{ZZ}(L,L) (L \overline{\varrho _{ZZ}(N,L)})  \notag \\ 
+ \overline{\varrho _{ZZ}(L,L)} \left(\varrho _{ZZZZ}(L,L,L,L) - 3 \varrho _{ZZZ}(L,L,\varrho _{ZZ} \cdot L)\right).
\end{align}
Plugging these into the formula for $R_{,11}$, we complete the proof.
\end{proof}

\subsection{The Cartan tensor $Q_{11}$} Plugging \cref{e:a110,e:a1111,e:r11} into \eqref{e:chenglee}, we to obtain 
\begin{theorem}\label{cartan} Let $M$ be given by $\varrho =0$ where $\varrho $ is given by \cref{e:perbsph} and let $\theta = \iota^{\ast}(i\bar{\partial} \varrho )$. Then Cartan tensor of $M$ is given by \eqref{e:26}, where the component $Q_{11}$ takes the following form
\begin{align*}
Q_{11}
= \sum_{k=2}^{4}\mathcal{Q}_k[\varrho ]\biggl|_M,
\end{align*}
where
\begin{align}
\mathcal{Q}_{2}[\varrho ]
=
\varrho _{ZZ}(L,L) \biggl(-\frac{1}{2} \frac{\overline{\varrho _{ZZ}(L,L)}\det \varrho _{ZZ}}{\J^3} - 2\frac{\overline{\varrho _{ZZ}(N,N)}}{\J^3} \notag  \\
+ \frac{|\varrho _{ZZ}(L,L)|^2}{\J^4} - 4\frac{|\varrho _{ZZ}(N,L)|^2}{\J^4}
-\frac{5}{2}\frac{\overline{\varrho _{ZZ}(L,L)}\left(\varrho _{ZZ}(N,L)\right)^2}{\J^5}\biggr),
\end{align}
\begin{align}
\mathcal{Q}_3[\varrho ]
=
-\frac{1}{2}\frac{|\varrho _{ZZ}(L,L)|^2 \varrho _{ZZZ}(N,L,L)}{\J^4} %\notag\\
+ \frac{1}{2} \frac{\overline{\varrho _{ZZ}(L,L)}\, \varrho _{ZZZ}(L,L,\varrho _{ZZ}\cdot L)}{\J^3} \notag \\
+ \frac{4}{3} \frac{\overline{\varrho _{ZZ}(N,L)}\, \varrho _{ZZZ}(L,L,L)}{\J^3}
+
\frac{7}{6} \frac{\overline{\varrho _{ZZ}(L,L)}\, \varrho _{ZZ}(N,L) \, \varrho _{ZZZ}(L,L,L)}{\J^4},
\end{align}
and 
\begin{equation}
\mathcal{Q}_4[\varrho ]
=
\frac{1}{6} \frac{\overline{\varrho _{ZZ}(L,L)}\,\varrho _{ZZZZ}(L,L,L,L)}{\J^3}.
\end{equation}
\end{theorem}
When restricted to $M$, $\J = |\varrho _z|^2 + |\varrho _w|^2$.

The simplest but interesting case where this theorem applies is arguably that of real ellipsoids. In this case, $f(z,w)$ is a quadratic polynomial and hence $\mathcal{Q}_3[\varrho] = \mathcal{Q}_4[\varrho] = 0$. Thus, the component $Q_{11}$ of the Cartan tensor equals $\mathcal{Q}_2[\varrho  ]$ which is the product of two factors.
\begin{corollary}\label{cor:49}
If $f$ is a quadratic polynomial, then the CR umbilical locus of $M$ is the locus of points on $M$ satisfying either
\begin{equation}\label{e:u1}
0= \varrho _{ZZ}(L,L),
\end{equation}
or 
\begin{align}\label{e:u2}
0= -\frac{1}{2} \frac{\overline{\varrho _{ZZ}(L,L)}\det \varrho _{ZZ}}{\J^3} - 2\frac{\overline{\varrho _{ZZ}(N,N)}}{\J^3} 
+ \frac{|\varrho _{ZZ}(L,L)|^2}{\J^4}   \notag  \\ 
- 4\frac{|\varrho _{ZZ}(N,L)|^2}{\J^4}
-\frac{5}{2}\frac{\overline{\varrho _{ZZ}(L,L)}\left(\varrho _{ZZ}(N,L)\right)^2}{\J^5}.
\end{align}
\end{corollary}
\section{The CR umbilical points of a real ellipsoid}
By Corollary \ref{cor:49}, the CR umbilical locus of an ellipsoid is determined by equations \eqref{e:u1} and \eqref{e:u2}. We first solve the simpler one \eqref{e:u1} and recover the curves $\gamma_{1,2}$ found in \cite{foo2018parametric}.
\begin{proposition} If $0\leqslant b \leqslant  a <1$ and $a > 0$, then the solution of the system of equations
\[
\varrho  = 0, \quad \varrho _{ZZ}(L,L)
\]
is given by
\begin{align*}
z 
& =
\sqrt{\frac{a}{a+b}} \left(\sqrt{\frac{1-b}{1+a}}\cos (t) + i \sqrt{\frac{1+b}{1-a}} \sin(t)\right),\\
w
& = \pm\sqrt{\frac{b}{a+b}} \left(\sqrt{\frac{1-a}{1+b}}\sin (t) - i \sqrt{\frac{1+a}{1-b}} \cos(t)\right).
\end{align*}
\end{proposition}
\begin{proof} Let
\begin{equation}\label{e:ell}
\varrho (z,w,\zba,\wba)
:=
- 1 + |z|^2 + |w|^2 + \Re (az^2 + bw^2)  = 0.
\end{equation}
We have $\varrho _z = \zba + az$, $\varrho _w = \wba + b w$, $\varrho _{zz} = a, \varrho _{ww} = b$, and $\varrho _{zw} = 0$. Thus, when being restricted to $M$, 
\begin{equation}
\varrho _{ZZ}(L,L) = - a \varrho _w^2 - b \varrho _z^2
\end{equation}
is a homogeneous polynomial of degree 2 in $z,w,\zba$, and $\wba$. The affine real algebraic variety $\mathcal{X}:=\{\varrho _{ZZ}(L,L)=0\}$ is a conical surface in $\mathbb{C}^2 \cong \mathbb{R}^4$. In fact, $\mathcal{X}$ is the intersection of two conical quadrics with apexes lie at the origin. Thus $\mathcal{X} \cap \mathcal{E} \ne \emptyset$, provided that $\mathcal{X}$ is not a single point.

In real coordinates $(x,y,u,v)$, with $z = x + iy$ and $w = u+iv$, we have
\begin{align*}
\Re \varrho _{ZZ}(L,L)
& =
-a (b+1)^2 u^2 + a (b-1)^2 v^2 - b(1 + a)^2 x^2 + b(a-1)^2 y^2,\\
\Im \varrho _{ZZ}(L,L)
& =
-2 a \left(b^2-1\right) u v -2b \left(a^2-1\right)  x y.
\end{align*}
If $a = 0$ and $b\ne 0$, then $z=0$. Similarly, if $b = 0$ and $a\ne 0$, then $w=0$.
If $a\ne 0$ and $b\ne 0$, we introduce $\tilde{u} = \sqrt{a} (b+1)u, \tilde{v} = \sqrt{a}(b-1) v, \tilde{x} = \sqrt{b}(a+1)x$, and $\tilde{y} = \sqrt{b} (a-1)y$. The equations $\Re \varrho _{ZZ}(L,L) = \Im \varrho _{ZZ}(L,L) =0$ reduce to
\begin{align*}
\tilde{u}^2 - \tilde{v}^2 & = \tilde{y}^2 - \tilde{x}^2,\\
\tilde{u}\tilde{v} & = - \tilde{x} \tilde{y}.
\end{align*}
Thus, either $\tilde{u} = \tilde{y}$ and $\tilde{v} = -\tilde{x}$, or $\tilde{u} = - \tilde{y}$ and $\tilde{v} = \tilde{x}$. These two cases are similar and we will only consider the first case and leave the detail of the other case to the readers. In fact, plugging these into $\varrho $, we obtain
\begin{equation}\label{e:43}
\frac{\tilde{x}^2}{(1+a)(1-b)} + \frac{\tilde{y}^2}{(1-a)(1+b)} = \frac{ab}{a+b}.
\end{equation}
Equation \eqref{e:43} determines an ellipse in the plane $\RR^2_{ \tilde{x},\tilde{y}}$ which has a well-known parametrization and we can easily conclude the proof.
\end{proof}
Equation \eqref{e:u2} is more complicated and we can only solve it explicitly in two special cases. At the end, we will show that \eqref{e:u2} has a non-empty solution set in general.
\subsection{The case $b=0$} So $M$ is an ellipsoid of revolution given by
\begin{equation}
\varrho := -1 + |z|^2 + |w|^2 + \Re (az^2) = 0.
\end{equation}
Since $\varrho _{ZZ}(L,L) = a\varrho _w^2 = a \wba^2$, we find that $\varrho_{ZZ}(L,L) = 0$ if and only if
\begin{equation}
w = 0,\ |z|^2 + \Re (a z^2) = 1,
\end{equation}
or 
\begin{equation}
w = 0, \
z = \sqrt{\frac{1}{1+a}} \cos(t) + i \sqrt{\frac{1}{1-a}} \sin (t), 
\quad
t\in [0,2\pi).
\end{equation}
As mentioned earlier, the fact that $M$ is umbilical along this locus follows immediately from Chern--Moser normal form \cite{chern1974real}.

To determine the whole CR umbilical locus, we also compute
\begin{equation}
\varrho _{ZZ}(N,L) = a \varrho _w \varrho _{\zba},
\quad 
\varrho _{ZZ}(N,N) = a\varrho _{\zba}^2, 
\quad 
\det \varrho _{ZZ} = 0.
\end{equation}
Denote by $\mathcal{P}[\varrho]$ the right-hand side of \eqref{e:u2}. Our goal is to solve 
\[
\mathcal{P}[\varrho] = 0.
\]
By direct calculation, we can factor
\begin{equation}
\Im \mathcal{P}[\varrho]
=
a \Re(\varrho _z) \Im (\varrho _z) ((5a^2 - 4)|w|^4 - 4 (|\varrho _z|^4+2|w|^2|\varrho _z|^2)).
\end{equation}
\textit{Case 1:} $\Im (\varrho _z) = 0$.

Then $z$ is real and $\varrho _z = \varrho _{\zba} \in \mathbb{R}$. Plugging this into $\Re \mathcal{P}[\varrho]$, we have
\begin{equation}
\Re \mathcal{P}[\varrho]
=
\frac{a}{4}|w|^6 \left(4\tau^6 +(8a +8) \tau^4 + (4+6a + 5a^2) \tau^2 - 2a\right),
\end{equation}
where
\begin{equation}
\tau = \frac{\varrho _z }{|w|}.
\end{equation}
Thus, that $\Re \mathcal{P}[\varrho] = 0$ gives
\begin{equation*}
(1+a) z = \varrho _z = \tau\, |w|,
\end{equation*}
where $\tau^2 = s_0$ is the unique positive solution (which belongs to $(0,a/2)$) to the cubic
\begin{equation}
4s^3 + 8 (1+a) s^2 + (4+6a + 5a^2) s - 2a =0.
\end{equation}
Plugging this into the defining function for the ellipsoid, we have
\begin{align}
z & = \pm \sqrt{\frac{s_0}{(1+a)(1+a+s_0)}},\\
w & = \sqrt{\frac{1+a}{1+a+s_0}} \left(\cos(t) + i \sin (t)\right).
\end{align}
We obtain two closed curves of CR umbilical points.

\noindent
\textit{Case 2:} $\Re \varrho _z = 0$. Then $\varrho _z = - \varrho _{\zba}$. Plugging into the first equation yields
\begin{equation}
\Re \mathcal{P}[\varrho] = -\frac{1}{4} a \left(2a - \left(5 a^2-6 a+4\right) \tau + (8-8 a) \tau ^2 -  \tau ^3\right) < 0
\end{equation}
for 
\begin{equation}
\tau = \frac{\varrho ^2_z}{|w|^2} < 0.
\end{equation}
Thus, this case does not give any eligible solution.

\noindent
\textit{Case 3:} $(5a^2 - 4)|w|^4 - 4 (|\varrho _z|^4+2|w|^2|\varrho _z|^2) = 0$. This implies
\begin{equation}
|\varrho _z|^2 = \left(\frac{a\sqrt{5}}{2}-1\right)|w|^2.
\end{equation}
Plugging this into $\mathcal{P}[\varrho]$, we have
\begin{equation}
\Re \mathcal{P}[\varrho]
=
-\frac{5}{2}a^3 |w|^4 \left(-2 a |w|^2+\sqrt{5} |w|^2 - 2\Re (\varrho _z^2)\right),
\end{equation}
which never vanishes.

Thus, the case $b=0$, the CR umbilical locus consists of exactly three closed curves. We finish the proof of Part (i) in Theorem~\ref{thm:1}.
\subsection{The case $a=b$} In this case, $M$ is an ellipsoid given by
\begin{equation*}
\varrho := -1 + |z|^2 + |w|^2 + a \Re (z^2+w^2 ) = 0, \quad 0<a<1.
\end{equation*}
To solve the equation $\mathcal{P}[\varrho] = 0$, we observe that when $a = b$, 
\begin{equation*}
\varrho _{ZZ}(L,L) = \overline{\varrho _{ZZ}(N,N)},
\end{equation*}
while
\begin{equation*}
\varrho _{ZZ}(N,L)^2 = \overline{\varrho _{ZZ}(N,L)^2} = - |\varrho _{ZZ}(N,L)|^2.
\end{equation*}
Using these equalities, we can simplify 
\begin{align}
\Im \mathcal{P}[\varrho]
=
\frac{1}{2}\J^{-5}\Im (\varrho _{ZZ}(L,L)) \left((a^2 - 4) \J^2 - 5|\varrho _{ZZ}(N,L)|^2\right).
\end{align}
Thus, $\Im \mathcal{P}[\varrho] = 0$ if and only if
\begin{equation}
\Im (\varrho _z^2 + \varrho _w^2) = \Im (\varrho _{ZZ}(L,L))/a = 0.
\end{equation}
If we write
\begin{equation*}
\varrho _z = \alpha + i\beta, \ \varrho _w = \gamma + i\delta,
\end{equation*}
then we have
\begin{equation*}
\alpha\beta = -\gamma\delta.
\end{equation*}
We first suppose that $\gamma \ne 0$. Put $\tau = \beta/\gamma$, then we have
\begin{equation*}
\delta = - \tau\alpha.
\end{equation*}
Under these conditions, we have
\begin{align*}
\varrho _{ZZ}(L,L) = \overline{\varrho _{ZZ}(L,L)} = 
\varrho _{ZZ}(N,N) = \overline{\varrho _{ZZ}(N,N)} = a(1-\tau^2)(\gamma^2 + \alpha^2),
\end{align*}
and
\begin{equation*}
\varrho _{ZZ}(N,L)^2 = \overline{\varrho _{ZZ}(N,L)^2} = - |\varrho _{ZZ}(N,L)|^2
=
-4 a^2 \tau ^2 \left(\alpha ^2+\gamma ^2\right)^2.
\end{equation*}
Plugging these into the formula for $\mathcal{P}[\varrho]$, we easily find that
\begin{equation}
\mathcal{P}[\varrho]
=
\frac{1}{2} a \left(\alpha ^2+\gamma ^2\right)^3 P(a,\tau)
\end{equation}
where 
\begin{equation*}
P(a,\tau)
=
(a^2+2a+4) \tau ^6+(4-a (19 a+34)) \tau ^4+(a (19 a-34)-4) \tau ^2-a^2+ 2 a-4.
\end{equation*}
To solve $P(a,\tau) = 0$, we put $s=\tau^2$ so that $P(a,\tau)=0$ becomes a cubic equation for $s$, which has a unique positive solution $s= s_0$. On the other hand, plugging these into the equation for the ellipsoid ($\varrho  = 0$), we find that
\begin{equation}
\alpha^2 + \gamma^2 = \frac{1-a^2}{a \tau ^2-a+\tau ^2+1}.
\end{equation}
The last equation defines a circle in the $(\alpha,\gamma)$-plane.  Omitting the details, we present here two parametrized solution curves given by
\begin{align*}
z 
& =
\frac{1}{\sqrt{1-a +\tau^2 (1+a)}}\left( \frac{\sqrt{1-a}\, \cos (t)}{\sqrt{1+a}}+ \frac{i\,\tau\,\sqrt{1+a}\, \sin (t)}{\sqrt{1-a}}\right),\\
w 
& = 
\frac{1}{\sqrt{1-a +\tau^2 (1+a)}}\left( \frac{-\sqrt{1-a}\, \sin (t)}{\sqrt{1+a}}+ \frac{i\,\tau\,\sqrt{1+a}\, \cos (t)}{\sqrt{1-a}}\right).
\end{align*}

The case $\gamma = 0$ gives either $\alpha = 0$ or $\beta = 0$. In the first sub-case $\alpha =0$ we have no solution, while in the second case $\beta = 0$, we obtain CR umbilical points which belong to the same curves as above. We omit the details.

The curves given by \eqref{e:f1} and \eqref{e:f2} in the case $b=a$ are the same as those in \eqref{e:16} and \eqref{e:17} with $\tau = \pm 1$. Thus, the CR umbilical locus consists of exactly four curves. This completes the proof of Part (ii) in Theorem~\ref{thm:1}.
\subsection{The generic case $0< b < a < 1$} As before, $\mathcal{P}[\varrho  ]$ denotes the right-hand side of \eqref{e:u2}, so that a part of the CR umbilical locus is given by $\mathcal{P}[\varrho  ]=0$. Taking the real and imaginary parts of $\mathcal{P}[\varrho  ]$ and expressing the results in terms of $X:= i \varrho  _z^2$ and $Y := i \varrho  _w^2$, we have (after elementary but tedious calculations)
\begin{align}
\Im \mathcal{P}[\varrho  ]
=
\frac{1}{2} a(b^2-4) |X|^2 \Re(X) + \frac{1}{2} b(a^2-4) |Y|^2 \Re(Y) \notag \\
+ \frac{5}{2} a^2 b \Re(X^2 \overline{Y}) + \frac{5}{2} a b^2 \Re(Y^2 \overline{X}) \notag \\ - 4a (b^2 + 1) |XY| \Re(X) 
- 4b (a^2+1) |XY| \Re(Y) \notag \\ + \frac{1}{2}b(a^2 + 5b^2 -4) |X|^2 \Re(Y) + \frac{1}{2}a(b^2 + 5a^2 -4) |Y|^2 \Re(X) .
\end{align}
Clearly, when $\Re(X) = \Re(Y) = 0$, we have $\Im \mathcal{P}[\varrho  ] = 0$. The formula for $\Re \mathcal{P}[\varrho  ]$ is also symmetrical with respect to $a,b$ and $X,Y$. Precisely, we can write
\[
\Re \mathcal{P}[\varrho  ] = P(a,b, X,Y) + P(b,a,Y,X),
\]
for
\begin{align}\label{e:4.26}
P(a,b, X,Y) = \frac{1}{2} a(b^2 +4) |X|^2 \Im (X) + \frac{5}{2} a^2b \Im(X^2 \overline{Y}) - b^2 |X|^3 \notag \\ 
+ 4a(1- b^2) |XY| \Im (X) + (4a^2 + 3b^2) |X^2 Y| \notag \\
- 10 ab |X| \Re (X \overline{Y}) + \frac{1}{2} b(4+a^2 + 5b^2) |X|^2 \Im (Y),
\end{align}
and, of course, $P(b,a,Y,X)$ is obtained from $P(a,b,X,Y)$ by exchanging $a$ and $b$ as well as $X$ and $Y$.

The equations $\Re \mathcal{P}[\varrho  ] = \Im \mathcal{P}[\varrho  ] = 0$ are real homogeneous sextic equations of four real variables $x,y,u,v$, each of them defines a conical real variety with vertex at the origin. To show that these two intersect, we consider the points satisfying $\Re(X) = \Re(Y) = 0$, which implies $\Im \mathcal{P}[\varrho  ] = 0$. To analyze the equation $\Re \mathcal{P}[\varrho  ] = 0$, we write $X = is$ and $Y = it$, for $s,t\in \RR$, so that $|X|= s, \ \Im(X) = s, \ \Im(X^2 \overline{Y}) = s^2 t$, and so on, and plug them into \eqref{e:4.26}. We obtain a homogeneous cubic equation for $s$ and $t$. Dividing by $t^3$, we obtain a cubic equation of one real variable for $\tau := s/t$, which has a unique real solution, as can be checked directly. We leave details to the readers. Thus, the real conical variety $\{(z,w) \in \CC^2 \colon \Re \mathcal{P}[\varrho  ] = \Im \mathcal{P}[\varrho  ] = 0\}$ is not a single point and hence it must intersect the ellipsoid since its vertex is interior to the ellipsoid. This shows that $\mathcal{V} = \mathcal{E} \cap \{(z,w) \in \CC^2 \colon \Re \mathcal{P}[\varrho  ] = \Im \mathcal{P}[\varrho  ] = 0\}$ is a non-trivial conical real variety. In view of Parts (i) and (ii), for ``generic'' values of $a$ and $b$, $\mathcal{V}$ does not coincide with $\gamma_{1,2}$. Hence the proof of Theorem~\ref{thm:1} is complete. \hfill \qedsymbol

\medskip

We conclude this paper by briefly discussing a related and interesting notion of umbilicity for real hypersurfaces in $\mathbb{C}^2$ (or $\mathbb{C}P^2$). A point $p\in M$ is called a projective-umbilical point of $M$  if there is a projective image of the unit sphere with third-order (or better) contact with $M$ at $p$. This notion was defined and studied recently by Barrett--Grundmeier in \cite{barrett2020projective}. In that paper, it is proved that on compact circular real hypersurfaces in $\CC^2$ projective-umbilical points must exist. On such hypersurfaces, the existence of CR umbilical points was proved earlier by Ebenfelt and the author \cite{ebenfelt2017umbilical}. Thus, it is natural to ask if every real ellipsoid in $\mathbb{C}^2$ admits projective-umbilical points (i.e., a version of Huang--Ji's theorem for projective-umbilicality)? The answer for this question turns out to be complete and related to our computation: The curves $\gamma_{1,2}$ (solution to $\varrho_{ZZ}(L,L) = 0$) are precisely the projective-umbilical locus of $\mathcal{E}$, because the Beltrami-tensor, which characterized the projective-umbilicity, is given by
\[
\mathcal{B}_{\mathcal{E}} = -\frac{\varrho_{ZZ}(L,L)}{\J} \cdot \frac{dz \wedge dw}{d\zba \wedge d\wba}.
\]
See \cite{barrett2020projective} for the details. On the other hand, since the variety $\mathcal{V}$ is generally different from the curves $\gamma_{1,2}$ (as evident from Theorem~\ref{thm:1}, Parts (i) and (ii)), the CR umbilical points lying on $\mathcal{V}\subset \E $ are generally not projective-umbilic.

\end{document}